\documentclass[a4paper]{article}
\usepackage{fullpage}
\usepackage{amsmath}
\usepackage{amssymb}
\usepackage{amsthm}
\usepackage{graphicx}
\usepackage{bbm}
\usepackage{enumerate}
\usepackage{microtype}
\usepackage{xcolor}
\usepackage[none]{hyphenat}
\usepackage[colorlinks=true,urlcolor=blue,linkcolor=blue,plainpages=false,pdfpagelabels]{hyperref}
\allowdisplaybreaks

\newtheorem{theorem}{Theorem}[section]
\newtheorem{proposition}[theorem]{Proposition}
\newtheorem{lemma}[theorem]{Lemma}

\newtheorem{definition}[theorem]{Definition}

\newcommand{\N}{\mathbb{N}}

\newcommand{\eps}{\varepsilon}
\newcommand{\vp}{\varphi}

\title{Quantitative inconsistent feasibility for averaged mappings}
\author{Andrei Sipo\c s${}^{a,b}$\\[2mm]
\footnotesize ${}^a$Research Center for Logic, Optimization and Security (LOS), Department of Computer Science,\\
\footnotesize Faculty of Mathematics and Computer Science, University of Bucharest,\\
\footnotesize Academiei 14, 010014 Bucharest, Romania\\[1mm]
\footnotesize ${}^b$Simion Stoilow Institute of Mathematics of the Romanian Academy,\\
\footnotesize Calea Grivi\c tei 21, 010702 Bucharest, Romania \\[2mm]
\footnotesize E-mail: andrei.sipos@fmi.unibuc.ro\\
}
\date{}

\begin{document}

\maketitle

\begin{abstract}
Bauschke and Moursi have recently obtained results that implicitly contain the fact that the composition of finitely many averaged mappings on a Hilbert space that have approximate fixed points also has approximate fixed points and thus is asymptotically regular. Using techniques of proof mining, we analyze their arguments to obtain effective uniform rates of asymptotic regularity.


\end{abstract}

\section{Introduction}\label{intro}

A fundamental issue in nonlinear analysis and optimization is {\it asymptotic regularity}: the property that, given a metric space $X$, a mapping $T$ on $X$ and a sequence $(x_n)$ in $X$ associated in some way to $T$,
$$\lim_{n \to \infty} d(x_n,Tx_n) = 0,$$
that is, $(x_n)$ is an approximate fixed point sequence for $T$. If, for any $x \in X$, the property holds for the Picard iteration of $T$ starting with $x$, i.e. $(T^nx)_n$, then we simply say that $T$ is {\it asymptotically regular}, as this is how the notion was originally introduced in \cite{BroPet66}. For a general iterative sequence, showing asymptotic regularity is often the first step in proving its (weak or strong) convergence.

Let now $X$ be a Hilbert space, $n \geq 1$ and $C_1,\ldots,C_n$ be closed, convex, nonempty subsets of $X$ with
$$\bigcap_{i=1}^n C_i \neq \emptyset.$$
This configuration is known as a {\it (consistent) convex feasibility problem}. Then -- denoting, for any closed, convex, nonempty subset $C$ of $X$, the metric projection onto it by $P_C$ -- the Picard iteration of the composition of projections $T:= P_{C_n} \circ\ldots\circ P_{C_1}$ starting from any point $x \in X$ is weakly convergent to a point in the intersection of the sets, a classical result of Bregman \cite{Bre65}. For the more general problem of {\it inconsistent feasibility} -- where we do not assume that the intersection is nonempty -- it was hypothesized in \cite{BauBorLew97} that the asymptotic regularity of $T$ still holds. This was later proven by Bauschke \cite{Bau03}, by first showing that the mapping has arbitrarily small displacements (hence the name `zero displacement conjecture' for the hypothesis) using a number of {\it ad hoc} constructions on a cartesian power of the Hilbert space, and then invoking the fact that this is equivalent to asymptotic regularity, using that $T$ is strongly nonexpansive, a class of mappings introduced in \cite{BruRei77} that is closed under composition and contains the projection operators. The result was later generalized from projections to firmly nonexpansive mappings, first assumed in \cite{BauMarMofWan12} to have approximate fixed points and then dropping that requirement in \cite{BauMou18}, where one can only hope to get an upper bound on the minimal displacement vector of the composition mapping.

Recently, a massive generalization of the latter result from firmly nonexpansive to averaged mappings was obtained by Bauschke and Moursi in \cite{BauMouXX}. This larger class (for more information, see \cite{Com04}) still sits inside the class of strongly nonexpansive mappings and is closed under composition. In addition, the proof techniques are much more natural, making direct use of the properties such as cocoercivity and rectangularity of the monotone operators associated to the averaged mappings under discussion. In particular, for the case where each mapping has a minimal displacement vector equal to zero, this gives a new way of showing asymptotic regularity for the composition.

The question may also be approached in a quantitative way, i.e. one can ask for a {\it rate of asymptotic regularity} for $(x_n)$ with respect to $T$, which is a function $\Sigma: (0,\infty) \to \N$ such that for all $\eps > 0$ and all $n \geq \Sigma(\eps)$, $d(x_n,Tx_n) \leq\eps$. This ties into the area of proof mining \cite{Koh08}, an applied subfield of mathematical logic that concerns itself with finding additional (for example, quantitative) information in concrete mathematical proofs by analyzing them using tools from proof theory. As it may be seen e.g. in the recent survey of Kohlenbach \cite{Koh19}, proof mining has been highly successful in the last two decades at extracting rates of asymptotic regularity for widely used iterations of nonlinear analysis. A few years ago, Kohlenbach has analyzed the results in \cite{Bau03,BauMarMofWan12} presented above, extracted bounds for the approximate fixed points and by combining them with his previous results in \cite{Koh16} on strongly nonexpansive mappings, obtained rates of asymptotic regularity \cite{Koh19b}. Although the analyzed proofs are highly non-trivial, appealing to deep results such as Minty's theorem, the resulting rate is of surprisingly low complexity.

In this paper, we update the techniques in \cite{Koh16,Koh19b} in order to analyze \cite{BauMouXX} and give a rate of asymptotic regularity for the composition of averaged mappings. The essential ingredients of the proof in \cite{BauMouXX} are the facts that an averaged mapping is the reflected resolvent of a cocoercive operator (as shown by \cite[Proposition 2.2]{BauMouXX}, based on work in \cite{Gis17,MouVan19}) and that cocoercive operators are rectangular. For the latter we give a quantitative version in Proposition~\ref{prop-theta}. This we use then to get upper bounds on approximate fixed points of the composition of two averaged mappings in Theorem~\ref{thm-phi} and then, by induction, of finitely many averaged mappings in Theorem~\ref{thm-psi}. By computing the modulus of strong nonexpansiveness of averaged mappings in Proposition~\ref{av-sne} and applying the result previously obtained in \cite{Koh19b}, Theorem~\ref{vp}, which quantitatively links the zero displacement of a strongly nonexpansive mapping with its asymptotic regularity, we obtain the desired rate in Theorem~\ref{thm-main}.

\section{Main results}\label{sec:main}

We start with some preliminaries. Let $X$ be a Hilbert space. A mapping $T: X \to X$ is called {\it nonexpansive} if for all $x$, $y \in X$, $\|Tx-Ty\| \leq \|x-y\|$. If $\alpha \in (0,1)$, a mapping $R:X \to X$ is called {\it $\alpha$-averaged} if there is a nonexpansive mapping $T:X \to X$ such that for all $x \in X$, $Rx = (1-\alpha)x+\alpha Tx$. Every averaged mapping is clearly nonexpansive; in particular, a $(1/2)$-averaged operator is called {\it firmly nonexpansive}, so there is a bijective correspondence between firmly nonexpansive and plainly nonexpansive operators given by $U \mapsto 2U - id_X$. By \cite[Proposition 4.4]{BauCom17}, an operator $U$ is firmly nonexpansive if and only if for all $x$, $y \in X$, $\|Ux-Uy\|^2 \leq \langle x-y,Ux-Uy\rangle$.

For any $\alpha$, $\beta \in(0,1)$, we define $\alpha \star \beta$ to be equal to
$$\frac{\alpha+\beta-2\alpha\beta}{1-\alpha\beta} = \frac1{1+\frac1{\frac\alpha{1-\alpha} + \frac\beta{1-\beta}}}.$$
Using the expression in the right-hand side, we may immediately derive that this operation is associative and commutative and that for any $m \geq 2$ and any $\alpha_1,\ldots,\alpha_m \in (0,1)$,
$$\alpha_1 \star\cdots \star\alpha_m = \frac1{1+\frac1{\sum_{i=1}^m \frac{\alpha_i}{1-\alpha_i}}}.$$
By \cite[Proposition 4.46]{BauCom17}, for any $m \geq 2$, $\alpha_1,\ldots,\alpha_m \in (0,1)$ and $R_1,\ldots,R_m:X \to X$ such that for each $i$, $R_i$ is $\alpha_i$-averaged, one has that $R_m \circ\ldots\circ R_1$ is $(\alpha_1 \star\cdots \star\alpha_m)$-averaged.

A set-valued operator $A \subseteq X \times X$ is called {\it monotone} if for any $(a,b)$, $(c,d) \in A$, $\langle a-c,b-d \rangle \geq 0$; it is {\it maximally monotone} (or {\it maximal monotone}) if it is maximal among monotone operators as ordered by inclusion. It is obvious that if $A$ is (maximally) monotone, then $A^{-1}$ is also (maximally) monotone. If $A$ is maximally monotone, then $(id_X + A)^{-1}$ is a firmly nonexpansive single-valued mapping on $X$ which is denoted by $J_A$ and called the {\it resolvent} of $A$. This association is bijective, by \cite[Propositions 23.8 and 23.10]{BauCom17}, and if we compose it with the previous bijection, we obtain the {\it reflected resolvent} of $A$, $R_A:=2J_A - id_X$.

Let $\beta > 0$. A set-valued operator $A \subseteq X \times X$ is called {\it $\beta$-cocoercive} if for any $(a,b)$, $(c,d) \in A$, $\langle a-c,b-d \rangle \geq \beta\|b-d\|^2$. A $\beta$-cocoercive operator $A$ is necessarily a single-valued mapping on the whole of $X$, since on the one hand single-valuedness is trivially implied by the definition, whereas on the other hand the condition is equivalent to the fact that $A^{-1}$ is strongly monotone with constant $\beta$, and thus -- by \cite[Proposition 22.11]{BauCom17} -- surjective, yielding that $A$ has full domain. Therefore, such a single-valued mapping $A:X \to X$ is $\beta$-cocoercive if and only if for all $x$, $y \in X$, $\langle x-y,Ax-Ay \rangle \geq \beta\|Ax-Ay\|^2$, i.e. if and only if $\beta A$ is firmly nonexpansive, which implies that $A$ is $(1/\beta)$-Lipschitz. By \cite[Proposition 2.2]{BauMouXX}, if $A \subseteq X \times X$ is maximally monotone, then $A$ is $\beta$-cocoercive if and only if $R_A$ is $(1+\beta)^{-1}$-averaged.

A set-valued operator $A \subseteq X \times X$ is called {\it rectangular} -- or {\it $3^*$-monotone} -- if for any $c$ in the domain of $A$ and any $b'$ in the range of $A$, $\sup_{(a,a')\in A} \langle a-c,b'-a'\rangle <\infty$. We have -- see \cite[Exemple 2]{BreHar76} and \cite[Examples 25.15 and 25.20]{BauCom17} for proofs -- that cocoercive operators are rectangular, and the following proposition expresses this fact quantitatively.

\begin{proposition}\label{prop-theta}
Put, for all $\beta$, $L_1$, $L_2$, $L_3 > 0$,
$$\Theta(\beta,L_1,L_2,L_3):=(L_1+L_2) \left(L_3 + \frac{L_1+L_2+2\beta L_3 + \sqrt{L_1^2 + L_2^2 + 2L_1L_2 + 8\beta L_1L_3 + 4\beta L_2L_3}}{2\beta}\right).$$
Let $X$ be a Hilbert space. Let $\beta$, $L_1$, $L_2$, $L_3 > 0$ and $A: X \to X$ be $\beta$-cocoercive. Let $b$, $c \in X$ with $\|b\|\leq L_1$, $\|c\|\leq L_2$ and $\|Ab\| \leq L_3$. Then for all $a \in X$,
$$\langle a-c,Ab-Aa \rangle \leq \Theta(\beta,L_1,L_2,L_3).$$
\end{proposition}

\begin{proof}
Let $a \in X$. Put
$$\rho:=\frac{L_1+L_2+2\beta L_3 + \sqrt{L_1^2 + L_2^2 + 2L_1L_2 + 8\beta L_1L_3 + 4\beta L_2L_3}}{2\beta},$$
so $\Theta(\beta,L_1,L_2,L_3)=(L_1+L_2)(L_3 +\rho)$. If $\|Aa\| \leq \rho$, then, using the fact that $A$ is $\beta$-cocoercive, so $\langle b-a, Ab-Aa \rangle \geq \beta\|Ab-Aa\|^2 \geq 0$, we have that
$$\langle a-c,Ab-Aa \rangle \leq \langle b-c, Ab-Aa \rangle \leq (\|b\|+\|c\|)(\|Ab\|+\|Aa\|) \leq (L_1 + L_2)(L_3 + \rho).$$
In the case where $\|Aa\| \geq \rho$, using the definition of $\rho$ we get that
$$\|Aa\| \geq \frac{\|b\| + \|c\| + 2\beta\|Ab\| + \sqrt{\|b\|^2 + \|c\|^2 + 2\|b\|\|c\| + 8\beta\|b\|\|Ab\| + 4\beta\|c\|\|Ab\|}}{2\beta},$$
where the right hand side is the rightmost zero of the quadratic real function
$$z \mapsto \beta z^2 + (-2\beta\|Ab\| - \|b\| - \|c\|)z + \beta\|Ab\|^2 - \|Ab\| \|b\|,$$
whose leading coefficient is strictly positive. Thus, we have, letting $z:=\|Aa||$,
$$\beta \|Aa\|^2 + (-2\beta\|Ab\| - \|b\| - \|c\|)\|Aa\| + \beta\|Ab\|^2 - \|Ab\| \|b\| \geq 0,$$
so, using the Cauchy-Schwarz inequality,
\begin{align*}
\beta\|Aa-Ab\|^2 - \|Aa\| \|b\| - \|Ab\| \|b\| &= \beta(\|Aa\|^2 - 2\langle Aa,Ab \rangle + \|Ab\|^2) \\
&\ \ \ \ -(\|b\|+\|c\|)\|Aa\| - \|Ab\| \|b\| + \|c\|\|Aa\| \\
&\geq \beta(\|Aa\|^2 - 2\|Aa\| \|Ab\| + \|Ab\|^2) \\
&\ \ \ \ -(\|b\|+\|c\|)\|Aa\| - \|Ab\| \|b\| + \|c\|\|Aa\| \\
&\geq \|c\|\|Aa\|.
\end{align*}
In addition, since $A$ is $\beta$-cocoercive,
$$\langle Aa-Ab, a-b \rangle \geq \beta \|Aa-Ab\|^2,$$
so
$$\langle Aa-Ab,a \rangle \geq \beta \|Aa-Ab\|^2 + \langle Aa-Ab,b \rangle \geq \beta \|Aa-Ab\|^2 - \|Aa\| \|b\| - \|Ab\| \|b\| \geq \|c\| \|Aa\|.$$
On the other hand,
$$\langle Aa-Ab,c \rangle \leq \|c\|(\|Aa\|+\|Ab\|),$$
so
$$\langle Aa-Ab, a-c \rangle \geq - \|c\|\|Ab\|,$$
i.e.
$$\langle a-c, Ab-Aa \rangle \leq \|c\|\|Ab\| \leq L_2L_3 \leq (L_1+L_2)(L_3+\rho).$$
\end{proof}

We may make use now of the above proposition to obtain a quantitative version of \cite[Theorem 3.3]{BauMouXX}, which uses an analysis of the argument used to prove the Br\'ezis-Haraux theorem in \cite{BreHar76}.

\begin{theorem}\label{thm-phi}
Let $\Theta$ be defined as in Proposition~\ref{prop-theta}. Put, for all $\alpha_1$, $\alpha_2 \in (0,1)$, $\delta > 0$ and $K : (0,\infty) \to (0, \infty)$,
\begin{align*}
B(\alpha_2,K,\delta)&:=\sqrt{\left(K\left(\frac\delta4\right)+\frac\delta8\right)^2 + 2\Theta\left(\alpha_2^{-1}-1,K\left(\frac\delta4\right)+\frac\delta8,K\left(\frac\delta4\right)+\frac\delta8,\frac\delta8\right)}\\
\Phi(\alpha_1,\alpha_2,K,\delta)&:=B(\alpha_2,K,\delta) \cdot \max\left(\sqrt{2}, \frac{4B(\alpha_2,K,\delta)}\delta \right) \cdot \frac1{1-\alpha_1} + \frac{\alpha_1}{1-\alpha_1} \left(K\left(\frac\delta4\right)+\frac\delta8\right) + \frac\delta8.
\end{align*}
Let $X$ be a Hilbert space. Let $\alpha_1$, $\alpha_2 \in (0,1)$ and $R_1$, $R_2:X \to X$ such that for each $i$, $R_i$ is $\alpha_i$-averaged. Put $R:=R_2 \circ R_1$. Let $K: (0,\infty) \to (0,\infty)$ be such that for all $i$ and all $\eps > 0$ there is a $p \in X$ with $\|p\| \leq K(\eps)$ and $\|p-R_ip\| \leq \eps$.

Then for all $\delta > 0$ there is a $p \in X$ with $\|p\| \leq \Phi(\alpha_1,\alpha_2,K,\delta)$ and $\|p-Rp\| \leq \delta$.
\end{theorem}

\begin{proof}
Let $\delta > 0$. We know that there are two single-valued maximal monotone operators, $A$, $B : X \to X$, with $A$ being $(\alpha_1^{-1}-1)$-cocoercive and $B$ being $(\alpha_2^{-1}-1)$-cocoercive, such that $R_1 = R_A$ and $R_2 = R_B$.

Put $\eps:=\delta/4$. By the hypothesis, there are $p$, $q\in X$ such that $\|p\|$, $\|q\| \leq K(\eps)$ and $\|p-R_Ap\|$, $\|q-R_Aq\| \leq \eps$. Since, by the definition of the reflected resolvent, $p-R_Ap = 2(p-J_Ap)$ and $q-R_Bq=2(q-J_Bq)$, we have that $\|p-J_Ap\|$, $\|q-J_Bq\| \leq \eps/2$. Also, we have, by the definition of the resolvent, that $p-J_Ap=AJ_Ap$ and $q-J_Bq=BJ_Bq$.

Put $f:=p-J_Ap+q-J_Bq$,
$$c:=\Theta\left(\alpha_2^{-1}-1, K(\eps)+\frac\eps2, K(\eps)+\frac\eps2, \frac\eps2\right),$$
and
$$\eta:=\min\left(\frac12,\frac{\eps^2}{\left(K(\eps)+\frac\eps2\right)^2 + 2c} \right),$$
so $\eta \in (0,1)$ and
$$\sqrt{\eta}\cdot \sqrt{\left(K(\eps)+\frac\eps2\right)^2 + 2c} \leq\eps.$$

By the sum rule, $A+B$ is maximally monotone. Then, by Minty's theorem, there is an $u \in X$ such that $f=\eta u + Au + Bu$. Since $A$ is monotone,
$\langle Au-(p-J_Ap), u-J_Ap \rangle \geq 0$. Since $B$ is $(\alpha_2^{-1}-1)$-cocoercive, and we know that
$$\|J_Bq\| \leq \|q\|+\|J_Bq-q\| \leq K(\eps) + \frac\eps2,$$
that similarly,
$$\|J_Ap\| \leq K(\eps) + \frac\eps2,$$
and that
$$\|BJ_Bq\| = \|q-J_Bq\| \leq \frac\eps2,$$
we may apply Proposition~\ref{prop-theta} to get that
$$\langle u - J_Ap, BJ_Bq - Bu \rangle \leq c,$$
so
$$\langle Bu - (q-J_Bq), u - J_Ap \rangle \geq -c.$$
Summing up, we get that
$$\langle f-\eta u - f, u - J_Ap \rangle \geq -c,$$
so
$$\langle u,u-J_Ap \rangle \leq c/\eta.$$
On the other hand, we have that
$$\|J_Ap\|^2 = \|u-(u-J_Ap)\|^2 = \|u\|^2 - 2\langle u,u-J_Ap \rangle + \|u-J_Ap\|^2 \geq \|u\|^2 - 2c/\eta + 0,$$
so
$$\|u\|^2 \leq \|J_Ap\|^2 + 2c/\eta$$
and
$$\eta \|u\|^2 \leq \eta \|J_Ap\|^2 + 2c \leq \|J_Ap\|^2 + 2c.$$
Therefore
$$\|\eta u\| = \sqrt{\eta} \cdot \sqrt{\eta} \cdot \|u\| \leq \sqrt{\eta} \cdot \sqrt{\|J_Ap\|^2 + 2c} \leq  \sqrt{\eta} \cdot \sqrt{\left(K(\eps) + \frac\eps2\right)^2 + 2c}  \leq \eps,$$
$$\|f-\eta u\| \leq \|f\| + \|\eta u\| \leq \frac\eps2 + \frac\eps2 + \eps = 2\eps.$$

We have that
$$2J_A - id_X = R_A = id_XR_A = (2J_B - R_B)R_A = 2J_BR_A - R_BR_A,$$
so
$$2J_A - 2J_BR_A = id_X -R_BR_A.$$
Now set
$$z:=(2J_A - 2J_BR_A)(u+Au) = (id_X -R_BR_A)(u+Au).$$
We have (using for the first equality the definition of the resolvent, and for the second the {\it inverse resolvent identity}, \cite[p. 399, (23.17)]{BauCom17}) that
$$Bu = J_{B^{-1}}(u+Bu) = u+Bu - J_B(u+Bu)$$
and, since by the definition of the resolvent, $J_A(u+Au)=u=J_B(u+Bu)$ and, by the definition of the reflected resolvent, $R_A(u+Au)=u-Au$,
$$z=2u-2J_B(u-Au)=2J_B(u+Bu) - 2J_B(u-Au).$$

We may now bound:
$$\|z\| = 2\|J_B(u+Bu) - J_B(u-Au)\| \leq 2\|u+Bu-u+Au\| = 2\|Au+Bu\| = 2\|f-\eta u\| \leq 4\eps = \delta.$$
We may then set $p:=u+Au$, since, as $z=p-Rp$, we have that $\|p-Rp\| \leq \delta$. We now only have to bound $p$.

Since we have seen that
$$\sqrt{\eta} \cdot \|u\| \leq B(\alpha_2,K,\delta)$$
we have, by the definition of $\eta$, that
$$\|u\| \leq \max\left(B(\alpha_2,K,\delta) \cdot\sqrt{2}, \frac{B^2(\alpha_2,K,\delta)}\eps\right).$$
Since $A$ is $(\alpha_1^{-1}-1)$-cocoercive, it is $\alpha_1/(1-\alpha_1)$-Lipschitz, so
$$\|Au\| \leq \|Au - AJ_Ap\| + \|AJ_Ap|| \leq \frac{\alpha_1}{1-\alpha_1}(\|u\| + \|J_Ap\|) + \|p-J_Ap\|.$$
By putting all bounds obtained so far together, we get that
$$\|p\| \leq \|u\| + \|Au\| \leq \Phi(\alpha_1,\alpha_2,K,\delta).$$
\end{proof}

The following is a quantitative version of \cite[Proposition 3.4]{BauMouXX}.

\begin{theorem}\label{thm-psi}
Let $\Phi$ be defined as in Theorem~\ref{thm-phi}. Define, for all $m \geq 2$, $\delta > 0$, $K : (0,\infty) \to (0, \infty)$ and suitable finite sequences $\{\alpha_i\}_i \subseteq (0,1)$,
\begin{align*}
\Psi(2,\{\alpha_i\}_{i=1}^2,K,\delta)&:=\Phi(\alpha_1,\alpha_2,K,\delta)\\
\Psi(m+1,\{\alpha_i\}_{i=1}^{m+1},K,\delta)&:=\Phi(\alpha_1\star\ldots\star\alpha_m,\alpha_{m+1},\rho\mapsto\max(\Psi(m,\{\alpha_i\}_{i=1}^m,K,\rho),K(\rho)) ,\delta)
\end{align*}
Let $X$ be a Hilbert space. Let $m \geq 2$, $\alpha_1,\ldots,\alpha_m \in (0,1)$ and $R_1,\ldots,R_m:X \to X$ such that for each $i$, $R_i$ is $\alpha_i$-averaged. Put $R:=R_m \circ\ldots\circ R_1$. Let $K: (0,\infty) \to (0,\infty)$ be such that for all $i$ and all $\eps > 0$ there is a $p \in X$ with $\|p\| \leq K(\eps)$ and $\|p-R_ip\| \leq \eps$.

Then for all $\delta > 0$ there is a $p \in X$ with $\|p\| \leq \Psi(m,\{\alpha_i\}_{i=1}^m,K,\delta)$ and $\|p-Rp\| \leq \delta$.
\end{theorem}

\begin{proof}
It follows by simple induction on $m$, using Theorem~\ref{thm-phi} for both the base step and the induction step and the fact that for each $l$, $R_l \circ \ldots\circ  R_1$ is $(\alpha_1\star\ldots\star\alpha_l)$-averaged.
\end{proof}

We may now return to the question of finding a rate of asymptotic regularity. Towards that end, we bring forward and expand upon quantitative results in \cite{Koh16,Koh19b} on strong nonexpansivity.

\begin{definition}
Let $X$ be a Hilbert space, $T:X\to X$ and $\omega:(0,\infty) \times (0,\infty) \to (0,\infty)$. Then $T$ is called {\bf strongly nonexpansive} with modulus $\omega$ if for any $b$, $\eps > 0$ and $x$, $y \in X$ with $\|x-y\| \leq b$ and $\|x-y\| - \|Tx-Ty\| < \omega(b,\eps)$, we have that $\|(x-y) - (Tx-Ty)\| < \eps$.
\end{definition}

\begin{theorem}[{cf. \cite[Theorem 1]{Koh19b}}]\label{vp}
Define, for any $\eps$, $b$, $d > 0$, $\alpha : (0, \infty) \to (0, \infty)$ and $\omega:(0,\infty) \times (0,\infty) \to (0,\infty)$,
$$\vp(\eps,b,d,\alpha,\omega):=\left\lceil \frac{18b + 12\alpha(\eps/6)}\eps-1\right\rceil \cdot \left\lceil \frac{d}{\omega\left(d,\frac{\eps^2}{27b + 18\alpha(\eps/6)}\right)} \right\rceil.$$
Let $X$ be a Hilbert space, $T:X\to X$ and $\omega:(0,\infty) \times (0,\infty) \to (0,\infty)$ such that $T$ is strongly nonexpansive with modulus $\omega$. Let $\alpha : (0, \infty) \to (0, \infty)$ such that for any $\delta > 0$ there is a $p \in X$ with $\|p\|\leq\alpha(\delta)$ and $\|p-Tp\| \leq \delta$. Then for any $\eps$, $b$, $d>0$ and any $x \in X$ with $\|x\| \leq b$ and $\|x-Tx\|\leq d$, we have that for any $n \geq \vp(\eps,b,d,\alpha,\omega)$, $\|T^nx-T^{n+1}x\|\leq \eps$.
\end{theorem}

The following lemma is the instantiation of \cite[Lemma 2.15]{Koh16} for Hilbert spaces, using their modulus of uniform convexity $\eps \mapsto \eps^2/8$.

\begin{lemma}\label{l-uc}
Let $X$ be a Hilbert space. Then for any $\eps \in (0,2]$, $d > 0$, $\alpha \in (0,1)$ and $x$, $y \in X$, with $\|x\|$, $\|y\| \leq d$, if
$$\|(1-\alpha)x+\alpha y\| > \left(1-2\alpha(1-\alpha)\frac{\eps^2}8\right) \cdot d,$$
then
$$\|x-y\| < \eps \cdot d.$$
\end{lemma}

\begin{proposition}\label{av-sne}
Define, for any $\alpha \in (0,1)$, $b$, $\eps > 0$,
$$\omega_\alpha(b,\eps):=\frac{1-\alpha}{4b\alpha} \cdot \eps^2.$$
Let $X$ be a Hilbert space, $\alpha \in(0,1)$ and $R: X\to X$ an $\alpha$-averaged mapping. Then $R$ is strongly nonexpansive with modulus $\omega_\alpha$.
\end{proposition}

\begin{proof}
Let $T:X \to X$ be nonexpansive such that $R=(1-\alpha)id_X + \alpha T$. Let $b$, $\eps > 0$ and $x$, $y \in X$ with $\|x-y\| \leq b$ and $\|x-y\| - \|Rx-Ry\| < \frac{1-\alpha}{4b\alpha} \cdot \eps^2$. We have to show that $\|(x-y) - (Rx-Ry)\| < \eps$.

If $\|x-y\| < \eps/2$, then $\|Rx-Ry\| < \eps/2$, so clearly $\|(x-y) - (Rx-Ry)\| < \eps$. Assume now that $\|x-y\| \geq \eps/2$, so $\eps/\|x-y\| \leq 2$.
We have that
$$\|Rx-Ry\| = \|(1-\alpha)(x-y) + \alpha(Tx-Ty) \|$$
and on the other hand
$$\|x-y\| - \|Rx-Ry\| < \frac{1-\alpha}{4b\alpha} \cdot \eps^2 \leq \frac{1-\alpha}{4\|x-y\|\alpha} \cdot \eps^2 = \frac{2(1-\alpha)}\alpha \cdot \frac{\eps^2}{8\|x-y\|^2} \cdot \|x-y\|,$$
so
$$\|(1-\alpha)(x-y) + \alpha(Tx-Ty) \| > \|x-y\| - 2\alpha(1-\alpha) \cdot \frac{(\eps/\alpha)^2}{8\|x-y\|^2} \cdot \|x-y\|.$$
Applying Lemma~\ref{l-uc} for $\eps \mapsto \eps/(\alpha\|x-y\|)$, $d \mapsto \|x-y\|$, $x \mapsto x-y$ and $y \mapsto Tx-Ty$, we get that
$$\|(x-y) - (Tx-Ty)\| < \frac\eps{\alpha\|x-y\|} \cdot \|x-y\| = \frac\eps\alpha,$$
so
$$\|(x-y) - (Rx-Ry)\| = \alpha\|(x-y) - (Tx-Ty)\| < \eps.$$
\end{proof}

Putting together the above results, we obtain the following.

\begin{theorem}\label{thm-main}
Let $\Psi$ be defined as in Theorem~\ref{thm-psi}, $\vp$ as in Theorem~\ref{vp} and $\omega_\bullet$ as in Proposition~\ref{av-sne}. Define, for all $m \geq 2$, $\eps$, $b$, $d > 0$, $K : (0,\infty) \to (0, \infty)$ and $\{\alpha_i\}_{i=1}^m \subseteq (0,1)$,
$$\Sigma_{m,\{\alpha_i\}_{i=1}^m,K,b,d}(\eps):=\vp(\eps,b,d,\delta\mapsto \Psi(m,\{\alpha_i\}_{i=1}^m,K,\delta),\omega_{\alpha_1\star\ldots\star\alpha_m}).$$
Let $X$ be a Hilbert space. Let $m \geq 2$, $\alpha_1,\ldots,\alpha_m \in (0,1)$ and $R_1,\ldots,R_m:X \to X$ such that for each $i$, $R_i$ is $\alpha_i$-averaged. Put $R:=R_m \circ\ldots\circ R_1$. Let $K: (0,\infty) \to (0,\infty)$ be such that for all $i$ and all $\eps > 0$ there is a $p \in X$ with $\|p\| \leq K(\eps)$ and $\|p-R_ip\| \leq \eps$.

Then for any $b$, $d>0$ and any $x \in X$ with $\|x\| \leq b$ and $\|x-Rx\|\leq d$, we have that $\Sigma_{m,\{\alpha_i\}_{i=1}^m,K,b,d}$ is a rate of asymptotic regularity for the sequence $(R^nx)$ w.r.t. $R$, i.e. for any $\eps>0$ and $n \geq \Sigma_{m,\{\alpha_i\}_{i=1}^m,K,b,d}(\eps)$,
$$\|R^nx-R^{n+1}x\|\leq \eps.$$
\end{theorem}

\section{Acknowledgements}

I would like to thank Ulrich Kohlenbach for pointing me to the paper \cite{BauMouXX} and for suggesting an improvement of Proposition~\ref{av-sne}.

This work has been supported by the German Science Foundation (DFG Project KO 1737/6-1) and by a grant of the Romanian Ministry of Research, Innovation and Digitization, CNCS/CCCDI -- UEFISCDI, project number PN-III-P1-1.1-PD-2019-0396, within PNCDI III.

\end{document}